\documentclass[11pt, a4paper, english]{amsart}
\usepackage{amsmath} % utilities for mathematics
\usepackage{amsthm} % utilities for theorem environment
\usepackage{amssymb} % loads mathematical fonts
\usepackage{amscd} % utilities for commutative diagrams
\usepackage{mathtools}
\usepackage{array} % core package

\usepackage[dvipsnames]{xcolor}
\usepackage[mathscr]{eucal} % loads a font accessible by \mathscr{} 
\usepackage[backref=page]{hyperref} % permits navigating on the pdf

\usepackage{caption} %
\usepackage{graphics,graphicx} % permits putting images
\usepackage{tikz,tikz-cd} % tool for diagrams
\usepackage{listings}
\usetikzlibrary{graphs,graphs.standard,calc}

\usetikzlibrary{decorations.pathreplacing,}%angles,quotes

\usepackage{enumerate} % permits enumerating using letters or other symbols

\usepackage{newtxtext,newtxmath}
\usepackage[margin=1.25in]{geometry}
\usepackage{verbatim}
\makeatletter
\@namedef{subjclassname@2020}{\textup{2020} Mathematics Subject Classification}
\makeatother

\DeclareMathAlphabet{\mathbbb}{U}{bbold}{m}{n}
\hypersetup{
    colorlinks = true,
    linkbordercolor = {white},
    linkcolor = {BrickRed},
    anchorcolor = {black},
    citecolor = {blue},
    filecolor = {BrickRed},
    menucolor = {red},
    runcolor = {BrickRed},
    urlcolor = {magenta}
}

\usetikzlibrary{automata}

\newtheoremstyle{teoremas}% <name>
{13pt}% <Space above>
{13pt}% <Space below>
{\itshape}% <Body font>
{}% <Indent amount>
{\bfseries}% <Theorem head font>
{}% <Punctuation after theorem head>
{.5em}% <Space after theorem headi>
{}% <Theorem head spec (can be left empty, meaning `normal')>

\theoremstyle{teoremas}
\newtheorem{teo}{Theorem}[section]

\newtheorem{cor}[teo]{Corollary}

\newtheorem{lemma}[teo]{Lemma}

\newtheorem{prop}[teo]{Proposition}

\newtheoremstyle{definition}% <name>
{12pt}% <Space above>
{12pt}% <Space below>
{}% <Body font>
{}% <Indent amount>
{\bfseries}% <Theorem head font>
{}% <Punctuation after theorem head>
{.5em}% <Space after theorem headi>
{}% <Theorem head spec (can be left empty, meaning `normal')>

\theoremstyle{definition}
\newtheorem{defi}[teo]{Definition}
\newtheorem{conj}[teo]{Conjecture}
\newtheorem{example}[teo]{Example}
\newtheorem{remark}[teo]{Remark}

\DeclareMathOperator{\rk}{rk}

\newcommand{\M}{\mathsf{M}}
\newcommand{\N}{\mathsf{N}}
\newcommand{\U}{\mathsf{U}}
\newcommand{\T}{\mathsf{T}}

\newcommand{\rank}{\operatorname{rk}}
\newcommand{\cl}{\operatorname{cl}}
\newcommand\size[1]{|#1|}
\title{The Merino--Welsh Conjecture for split matroids}

\author[L. Ferroni]{Luis Ferroni}
\thanks{}

\address{
  Department of Mathematics, KTH Royal Institute of Technology, Stockholm, Sweden
}
\email{ferroni@kth.se}

\author[B. Schr\"oter]{Benjamin Schr\"oter}
\address{
Institut für Mathematik, Goethe-Universität, Frankfurt, Germany
}
\email{schroeter@math.uni-frankfurt.de}


\allowdisplaybreaks

\begin{document}

\begin{abstract}
    In 1999 Merino and Welsh conjectured that evaluations of the Tutte polynomial of a graph satisfy an inequality.
    In this short article we show that the conjecture generalized to matroids holds for the large class of all split matroids by exploiting the structure of their lattice of cyclic flats. This class of matroids strictly contains all paving and copaving matroids.
\end{abstract}

\maketitle

\section{Introduction}
\noindent In 1999 Merino and Welsh \cite[Conjecture 7.1]{merino-welsh} posed the following conjecture.

\begin{conj}\label{conj:merino-welsh-graphs-max}
    For every connected graph $G$ without loops and bridges, the following inequality holds.
    \[
        \max(\alpha(G), \alpha^*(G)) \geq \tau(G),
    \]
    where $\tau(G)$ denotes the number of spanning trees of $G$,
    $\alpha(G)$ the number of acyclic orientations of~$G$, and 
    $\alpha^*(G)$ the number of totally cyclic orientations of $G$.
\end{conj}

The three invariants appearing in this conjecture are evaluations of the Tutte polynomial of $G$, which we denote by $T_G(x,y)$. More precisely, we have that $\tau(G)=T_G(1,1)$, $\alpha(G)=T_G(2,0)$ and $\alpha^*(G) = T_G(0,2)$. This point of view allows to extend the preceding conjecture to all matroids. Furthermore, Conde and Merino strengthened the Conjecture in \cite{conde-merino}, by proposing an ``additive'' and a ``multiplicative'' version of the conjecture for graphs. Generalized to matroids the three versions read as follows.

\begin{conj}\label{conj:mw}
    Let $\M$ be a matroid without loops and coloops. Then
    \begin{align}
        \max\left(T_{\M}(2,0), T_{\M}(0,2)\right) &\geq T_{\M}(1,1),\label{eq:merinoWelshMax}\\
        T_{\M}(2,0) + T_{\M}(0,2) &\geq 2\,T_{\M}(1,1),\label{eq:merinoWelshAdd}\\
        T_{\M}(2,0)\cdot T_{\M}(0,2) &\geq T_{\M}(1,1)^2.\label{eq:merinoWelsh}
    \end{align}
\end{conj}

An elementary argument and the non-negativity of the expressions above shows that \eqref{eq:merinoWelsh} implies~\eqref{eq:merinoWelshAdd}, which in turn implies \eqref{eq:merinoWelshMax}.

These inequalities have been discussed and established for specific classes of graphs and matroids. Particularly regarding graphs, Thomassen \cite{thomassen} showed that if a graph is either ``sparse'' or ``dense'' then it satisfies Conjecture \ref{conj:merino-welsh-graphs-max}. Furthermore, Noble and Royle \cite{noble-royle} proved that series-parallel graphs satisfy the multiplicative version of Conjecture \ref{conj:mw}, whereas Jackson \cite{jackson} proved several inequalities that have a resemblance to \eqref{eq:merinoWelsh}, but his method does not extend to our case.

Strictly within the scope of the matroidal version of the conjecture, Knauer et al.\ proved in \cite{knauer-martinez-ramirez} that the class of lattice path matroids satisfy the multiplicative version of Conjecture~\ref{conj:mw}. The key feature of their method is a clever way of setting up an induction where they manage to keep control over the base cases.

In \cite{chavez-merino-noble} Ch\'avez-Lomel\'i et al.\ proved a related conjecture \cite[Conjecture 2.4]{conde-merino} for all paving matroids without coloops. Namely, they showed that the Tutte polynomial of such a matroid is a convex function on the segment $x+y=2$ within the positive quadrant. This yields immediately that all such matroids satisfy the additive version~\eqref{eq:merinoWelshAdd} of the inequality. The relevance of this class of matroids comes from a conjecture \cite[Conjecture 1.6]{mayhew} posed by Mayhew et al.\ asserting that paving matroids are asymptotically predominant within all matroids.

The purpose of this short article is to show that the class of split matroids satisfies all of Conjecture \ref{conj:mw}. This class of matroids, introduced by Joswig and Schr\"oter in \cite{joswig-schroter} is minor closed and strictly larger than both the classes of paving and copaving matroids. It originates in the study of splits of hypersimplices, i.e., the base polytopes of uniform matroids, and the tropical Grassmannian. The main result of this note is the following.

\begin{teo}\label{teo:main}
    Let $\M$ be a split matroid without loops and coloops. Then
        \[ T_{\M}(2,0)\cdot T_{\M}(0,2) \geq T_{\M}(1,1)^2.\]
\end{teo}

Moreover, we identify the cases in which equality is attained. The method we employ is a modification of the arguments used in \cite{chavez-merino-noble} and \cite{knauer-martinez-ramirez}. A key feature of connected split matroids that we discuss below is a structural property of their lattice of cyclic flats that we use as definition.

\section{The toolbox}
\noindent Throughout the rest of the paper, we assume familiarity with basic concepts in matroid theory; we follow the terminology and notation of \cite{oxley} and point to \cite{bonin-demier} for further details on the lattice of cyclic flats.

\subsection{Split matroids} For the sake of simplicity in the exposition, we will use an alternative definition of split matroids using their cyclic flats. 

\begin{defi}[Split matroids]\label{def:splitmatroid}
A \emph{connected split matroid} is a connected matroid $\M$ whose proper cyclic flats form a clutter, that is if $E$ denotes the ground set of $\M$ and $F_1\subsetneq F_2$ are two cyclic flats then
either $F_1 = \varnothing$ or $F_2 = E$.

We call a matroid a \emph{split matroid} if it is isomorphic to a direct sum in which exactly one of the summands is a connected split matroid and the remaining summands are uniform matroids.
\end{defi}

\begin{remark}
    Let us explain why the above definition is equivalent to the original definition of split matroids. Observe that the only connected matroids with loops or coloops are $\U_{0,1}$ and $\U_{1,1}$. In particular if the matroid $\M$ is connected and has cardinality at least $2$, then it is loopless and coloopless. From the equivalence between (iii) and (iv) in \cite[Theorem~11]{BercziEtAl}, it follows that if $\M$ is connected and split it is either $\U_{0,1}$ or $\U_{1,1}$, or it is loopless and coloopless and the proper cyclic flats of $\M$ form a clutter. Conversely, if $\M$ is a connected matroid whose proper cyclic flats form a clutter, then by the same equivalence $\M$ has to be a connected split matroid. The second part of the definition that includes disconnected matroids is precisely \cite[Proposition~15]{joswig-schroter}.
\end{remark}

The name ``split matroid'' stems from a polyhedral point of view of matroids. The facet defining equations of the matroid base polytope of a split matroid are compatible splits of a hypersimplex, namely, the base polytope of a uniform matroid. In this article we focus only on combinatorial properties. For a thorough and detailed exploration of the interplay between the combinatorics and the geometry of split matroids we refer to \cite{ferroni-schroter}.

A fact that we will leverage is that the class of split matroids is closed under duality and taking minors. A complete list of excluded minors for split matroids can be found in \cite{cameron-mayhew} by Cameron and Mayhew, although we will not require it here.

A basic and important example of a split matroid is the following.

\begin{example}[Minimal matroids]\label{example:minimal}
    Consider the graphic matroid induced by a cycle of length $k+1$ where one edge is replaced by $n-k$ parallel copies of that edge (see Figure~\ref{fig:minimal}).
    We denote the corresponding matroid of rank $k$ and ground set of size $n$ by $\T_{k,n}$.
    \begin{figure}[ht]
        \centering
        \begin{tikzpicture}  
    	[scale=0.08,auto=center,every node/.style={circle, fill=black, inner sep=2.3pt},every edge/.append style = {thick}] 
    	\tikzstyle{edges} = [thick];
        \graph  [empty nodes, clockwise, radius=1em,
        n=9, p=0.3] 
            { subgraph C_n [n=5,clockwise,radius=1.7cm,name=A]};
            
		\draw[edges] (A 3) edge[bend right=20] (A 4);
		\draw[edges] (A 3) edge[bend right=-20] (A 4);
    \end{tikzpicture}\caption{The underlying graph of a minimal matroid $\T_{4,7}$}\label{fig:minimal}
\end{figure}
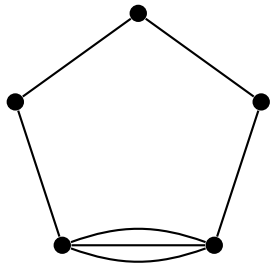
    In general, these matroids fail to be paving (or co-paving), but they are split matroids. 
    It is straightforward to check that the set of all the parallel edges is the only proper cyclic flat of the matroid $\T_{k,n}$ whenever $k>0$. 

    The matroid $\T_{k,n}$ is known in the literature by the name ``minimal matroid''.  Dinolt \cite{dinolt}  introduced that name for $\T_{k,n}$ since it is the only connected matroid (up to isomorphism) of rank $k$ and cardinality $n$ achieving the minimal number of bases $k(n-k)+1$. 
    
    Minimal matroids are lattice path matroids, which follows either by \cite[Theorem 3.2]{knauer-martinez-ramirez} or by \cite[Proposition 2.4]{ferroni2}. Moreover, notice that the dual of a minimal matroid is again a minimal matroid, that is $\T_{k,n}^* \cong \T_{n-k,n}$.
\end{example}

\subsection{Tutte polynomials and inequalities}

We will assume basic knowledge on Tutte polynomials of matroids. We refer to \cite{BrylawskiOxley:1992} for a thorough exposition regarding Tutte polynomials. In the following whenever we refer to the ``multiplicative Merino--Welsh conjecture'' we are speaking about the inequality Conjecture \ref{conj:mw}\eqref{eq:merinoWelsh}.

\begin{lemma}\label{lemma:duals}
    If $\M$ and $\N$ are matroids that satisfy the multiplicative Merino--Welsh conjecture, then both the dual matroid $\M^*$ and direct sum $\M\oplus \N$ satisfy the conjecture, too.
\end{lemma}
\begin{proof}
    This is immediate from the fact that Tutte polynomials fulfill the relation $T_{\M^*}(x,y) = T_\M(y,x)$ and $T_{\M\oplus\N}(x,y) = T_{\M}(x,y)\cdot T_{\N}(x,y)$.
\end{proof}

\begin{lemma}\label{lemma:deletion-contraction}
    Let $\M$ be a loopless and coloopless matroid on the ground set $E$. If there is an element $e\in E$ such that both $\M/e$ and $\M\setminus e$ satisfy the multiplicative Merino--Welsh Conjecture, then the conjecture is fulfilled by the matroid $\M$.
\end{lemma}

\begin{proof}
     See \cite[Lemma 2.2]{noble-royle} or \cite[Lemma 3.3]{knauer-martinez-ramirez}.
\end{proof}

\begin{remark}\label{rem:equality}
    The preceding two well-known results (more precisely, their proofs) imply that if $\M$ and $\N$ are matroids satisfying Conjecture \ref{conj:mw}\eqref{eq:merinoWelsh}, then $\M\oplus \N$ attains equality if and only if both $\M$ and $\N$ attain equality too. 
    Similarly, if $\M\setminus e$ and $\M/ e$ satisfy the multiplicative Merino--Welsh conjecture and for $\M$ the equality is attained, then $\M/ e $ and $\M\setminus e$ attain equality as well.
\end{remark}

\begin{lemma}\label{lemma:rank1}
    Every loopless and coloopless matroid of rank $1$ or corank $1$ satisfies the multiplicative Merino--Welsh conjecture.
\end{lemma}

\begin{proof}
    By Lemma \ref{lemma:duals} it suffices to consider only the case in which the rank is $1$. Observe that a loopless and coloopless matroid of rank $1$ is a uniform matroid $\M\cong \U_{1,n}$ for some $n\geq 2$. Its Tutte polynomial is $T_{\M}(x,y) = x + y + y^2 + y^3 + \cdots + y^{n-1}$. Thus, $T_{\M}(1,1) = n$, whereas $T_{\M}(2,0) = 2$ and $T_{\M}(0,2) = 2^n - 2$. Therefore, 
    \begin{align*}
        T_{\M}(2,0)\cdot T_{\M}(0,2) &= 2\cdot (2^n-2)
        \geq n^2 = T_{\M}(1,1)^2
    \end{align*}
    where equality holds only if $n=2$. This follows from the estimation $2\cdot(2^{n}-2) > 2^n \geq n^2$ for all $n > 3$ and $2\cdot(2^n-2) = 12 > 9 = n^2$ if $n=3$.
\end{proof}

We observe that as an alternative to our arguments the preceding result also follows either from \cite{noble-royle} or \cite{knauer-martinez-ramirez} given that all loopless and coloopless matroids of rank or corank $1$ are series-parallel matroids and lattice path matroids. 

\begin{lemma}\label{lemma:rank2}
     Every loopless and coloopless matroid of rank $2$ or corank $2$ satisfies the multiplicative Merino--Welsh conjecture. Furthermore, equality occurs only for the matroid  $\U_{1,2}\oplus \U_{1,2}$.
\end{lemma}

\begin{proof}
    Again, by Lemma \ref{lemma:duals} it suffices to consider only the case in which the matroid has rank~$2$. Observe that whenever $\M$ has rank $2$, cardinality $n$, and does not contain loops and coloops, we have that the two coefficients $[x^2y^0] T_{\M}(x,y)$ and $[x^0 y^{n-2}]T_{\M}(x,y)$ are both equal to $1$ (see \cite[Proposition 6.2.13(ii) and (v)]{BrylawskiOxley:1992}). Hence, we obtain $T_{\M}(2,0)\cdot T_{\M}(0,2) \geq 2^2\cdot 2^{n-2}=2^n$.
    
    On the other hand, since $T_{\M}(1,1)$ is the number of bases of $\M$, this expression is bounded from above by the binomial coefficient $\binom{n}{2}$.
    The inequality $\binom{n}{2}^2 < 2^n$ holds for every $n\geq 13$. Hence, for all rank $2$ matroids on $n\geq 13$ elements we have
        \begin{equation*}
            T_{\M}(2,0)\cdot T_{\M}(0,2) \geq 2^n > \binom{n}{2}^2 \geq T_{\M}(1,1)^2.
        \end{equation*}
    We used a computer to verify this inequality for all matroids of rank $2$ and size $n\leq 12$ that have neither a loop or coloop.
    The multiplicative Merino--Welsh conjecture holds for these cases, and equality is obtained only if $\M\cong \U_{1,2}\oplus\U_{1,2}$. The code for our computations can be found in the Appendix~\ref{appendix}.
\end{proof}

\begin{remark}
    The preceding proof relies on the fact that the exponential growth of the term $2^n$ guarantees that it eventually dominates the polynomial term $\tbinom{n}{2}^2$ for large values of $n$. This property easily extends to matroids of higher fixed rank. However, as the rank grows, one has to deal separately with a growing finite number of cases that could potentially yield a counterexample to Conjecture \ref{conj:mw}\eqref{eq:merinoWelsh}. 
\end{remark}

\section{The proof of Theorem \ref{teo:main}}

\begin{prop}\label{prop:base-case}
    Let $\M$ be a connected split matroid without loops and coloops. Assume that for every $e\in E$ in the ground set $E$ of $\M$ either the deletion $\M/e$ or the contraction $\M\setminus e$ has a loop or a coloop. Then $\M$ satisfies one of the following, or both:
    \begin{enumerate}[i)]
        \item The rank or corank of $\M$ is either one or two.
        \item $\M$ is isomorphic to a minimal matroid.
    \end{enumerate}
\end{prop}
\begin{proof}
    Let us assume that the rank of the matroid $\M$ is $k$ and the ground set $E$ is of size~$n$.
    We are in the first case of the statement unless $3\leq k\leq n-3$. Thus we may assume that those inequalities hold. 
    Notice that a contraction of a coloopless matroid cannot lead to a matroid with a coloop and similarly a deletion in $\M$ cannot create loops, thus the assumptions in the statement imply that for every element $e\in E$ we have either $\M/e$ has a loop or $\M\setminus e$ has a coloop. We assume without loss of generality that there is at least one element $e$ in the ground set such that $\M/e$ has loops because, otherwise, we can exchange our matroid by its dual $\M^*$ which satisfies all assumptions as well.
    
    Now, fix an element $e\in E$ such that $\M/e$ has loops. Consider $P=\cl_{\M}(e)$, the set of all the elements in $\M$ that are parallel to $e$. This is of course a proper cyclic flat of $\M$. If $j=|P|-1\geq 1$ then
    \[
        \M/e \cong \U_{0,j}\oplus \U_{k-1,n-j-1}.
    \]
    This is the case as $P\setminus e$ is the ground set of the first direct summand. To justify the fact that the second direct summand is a uniform matroid, notice that since $\M$ is connected and split, there is only one cyclic flat of $\M$ strictly containing $P$, namely the ground set $E$. Recall that a set $F$ is a flat of $\M/e$ if and only if $F\cup e$ is a flat of $\M$.
    Hence, in $\M/e$ the only cyclic flat which strictly contains the rank $0$ flat $P\setminus e$ is $E\setminus e$. It is evident that the contraction of $\M$ by any element in $P$ has loops. 
    
    Suppose that some element $f\not\in P$ has a parallel element in $\M$, that is $\cl_\M(f)$ is a cyclic flat, and thus the closure of the two cyclic flats $\cl_\M(e)$ and $\cl_\M(f)$ is $\cl_{\M}(\{e,f\})$ which is a cyclic flat too. 
    This is a contradiction to the assumption that $\M$ is a split matroid, because  $\varnothing\subsetneq P = \cl_\M(e) \subsetneq \cl_\M(\{e,f\}) \subsetneq E$ is a chain of cyclic flats; to see why the last inclusion is strict recall that we are assuming that
    the matroid~$\M$ has at least rank $3$ while $\cl_\M(\{e,f\})$ is of rank $2$.
    We conclude that $f\notin P$ cannot have parallel elements in $\M$; therefore the contraction $\M/f$ is loopless for all elements $f\not\in P$.
    
    In other words, so far we have shown that the only elements $f\in E$ such that $\M/f$ has loops are exactly the elements $f\in P=\cl_{\M}(e)$. Since $\M$ has rank at least $3$ and $\rk_{\M}(P) = 1$, we have that $E\setminus P\neq\varnothing$. 
    Furthermore, our assumptions on $\M$ imply that for all elements $f\in E\setminus P$ we have that $\M\setminus f$ must contain a coloop. 
    Rephrased in terms of the dual matroid $\M^*$ the last conclusion is equivalent to the fact that for all $f\in E\setminus P$, the contractions $\M^*/f$ have loops. By applying the same reasoning as above to the matroid $\M^*$ we may derive that $E\setminus P$ is a parallel class of $\M^*$.
    
    Using our findings we obtain
    \[
    1 = \rank_{\M^*}(E\setminus P) = \size{E \setminus P} + \rank_\M(P)- k = \size{ E\setminus P} + 1 - k.
    \]
    And thus $\size{E\setminus P} = k$; hence, it has to be a basis of $\M$. Further $(E\setminus P) \cup\{e\}$ has to be a circuit of the matroid~$\M$ as this matroid contains no coloops. It follows that $\M$ is a minimal matroid.
\end{proof}

Now we are prepared to prove Theorem~\ref{teo:main}, the main result of this paper.

\begin{proof}[Proof of Theorem~\ref{teo:main}]
    We proceed by strong induction on the cardinality of the ground set of our matroid. Our hypothesis is that whenever the cardinality of the ground set of a loopless and coloopless split matroid is smaller than $n$, then the multiplicative Merino--Welsh conjecture holds. 
    Fix a loopless and coloopless split matroid $\M$ whose ground set is of size $n$. We may assume that $\M$ is connected, as otherwise we apply Lemma \ref{lemma:duals} and the induction hypothesis to conclude that the matroid fulfills the desired inequality. This is because split matroids are minor closed, thus each direct summand is a split matroid on a smaller ground set.
    
    If there is an element of the ground set such that both $\M/e$ and $\M\setminus e$ are loopless and coloopless, then we use Lemma \ref{lemma:deletion-contraction} and the induction hypothesis. Here we are using the fact that both $\M/e$ and $\M\setminus e$ are loopless and coloopless split matroids. 
    Now assume that there is no such element. Since $\M$ is connected, we are able to use Proposition \ref{prop:base-case} to conclude that $\M$ has rank or corank smaller than $2$, or that it is a minimal matroid. In the first cases, Lemma \ref{lemma:rank1} and Lemma \ref{lemma:rank2} show that the inequality is satisfied. 
    In the case that the matroid is a minimal matroid we obtain the desired inequality as a corollary of either \cite[Theorem 4.5]{noble-royle} or alternatively \cite[Corollary 3.4]{knauer-martinez-ramirez} since every minimal matroid is both a series-parallel matroid and a lattice path matroid (cf.\ Example \ref{example:minimal}).
\end{proof}

The next two results are consequences of Theorem~\ref{teo:main} and the above  Lemmas.

\begin{cor}
    Let $\M$ be a loopless and coloopless matroid that is isomorphic to a direct sum of split matroids. Then the multiplicative Merino--Welsh conjecture holds true for $\M$.
\end{cor}

\begin{cor} 
    Let $\M$ be a direct sum of split matroids without loops and coloops, then
    \begin{equation}\label{eq:equality}
        T_{\M}(2,0)\cdot T_{\M}(0,2) = T_{\M}(1,1)^2
    \end{equation}
    if and only if every connected direct sum component is isomorphic to the uniform matroid $\U_{1,2}$. 
\end{cor}

\begin{proof}
    We follow the guidance of the main proof.
    As mentioned in Remark \ref{rem:equality}, a matroid~$\M$ which is a direct sum of split matroids satisfies equation~\eqref{eq:equality} if and only if all of its direct sum components do. Thus we may focus on connected split matroids.
    To conclude the proof it suffices to show that in the only connected split matroid attaining equality is $\U_{1,2}$.
    
    First notice that by Lemma~\ref{lemma:rank2} there exists no connected matroid of rank or corank $2$ satisfying equation~\eqref{eq:equality}. Similarly, Lemma~\ref{lemma:rank1} tells us that the only connected matroid of rank or corank $1$ that fulfills the equation is $\U_{1,2}$. In particular, the only two  loopless and coloopless matroids $\U_{1,3}$ and $\U_{2,3}$ on a ground set of cardinality $3$ do not satisfy the equation.
    Moreover, 
    it follows from \cite[Corollary 3.4]{knauer-martinez-ramirez} that 
    there is no minimal matroid on more than two elements for which equation~\eqref{eq:equality} holds, because minimal matroids are connected lattice path matroids.
    
    Now we rule out the existence of connected split matroids on a ground set of size $n\geq 4$ and satisfying \eqref{eq:equality}. 
    Let us assume that $\M$ is such a matroid whose ground set is of minimal size.
    Since $\M$ does not have rank/corank at most $2$ nor it is minimal, we conclude by Proposition \ref{prop:base-case} that there is some element $e\in E$ with the property that $\M\setminus e$ and $\M/ e$ are loopless and coloopless. Furthermore, since $\M$ is connected, by \cite[Proposition 4.3.1]{oxley} at least one of the two matroids $\M\setminus e$ and $\M/e$ is connected.
    
    All three matroids $\M$, $\M\setminus e$ and $\M/e$ are split and hence satisfy the multiplicative Merino--Welsh conjecture. Remark \ref{rem:equality} tells us that 
    the fact that $\M$ attains equality implies the same for both $\M/ e$ and $\M\setminus e$.
    It follows that at least one of them is a connected split matroid on $n-1$ elements fulfilling equation~\eqref{eq:equality} which contradicts our assumption, or otherwise $n=3$ which is not possible either.
\end{proof}

We conclude by relating our results to some previous work in the literature.

\begin{remark}
    It follows from the work of Ch\'avez-Lomel\'i et al.\ \cite{chavez-merino-noble} that the additive version of the Merino--Welsh conjecture holds for paving matroids. Notice that our Theorem \ref{teo:main} extends this result.
    However, it is not known whether the assignment 
    $x\mapsto T_{\M}(2-x,x)$ is a convex function on $[0,2]$ for every split matroid $\M$.
\end{remark}

\begin{remark} 
    Even if the general scheme of the proof of Theorem \ref{teo:main} is similar in spirit to that of \cite{noble-royle} and \cite{knauer-martinez-ramirez}, the classes of matroids are quite different.
    The families of split matroids and series-parallel matroids are essentially disjoint, their only common members on ground set of size at least~$2$ are precisely the minimal matroids, the graphic matroids whose simplification is $\U_{2,3}$ and their duals; this follows from \cite[Theorem 13]{BercziEtAl} as series-parallel matroids are graphic and hence binary. On the other hand, we also point out that neither of the classes of lattice path matroids and split matroids contain one another. 
    For example the Catalan matroid $C_3$ having rank $3$ and a ground set of size $6$ is a lattice path matroid that is not split because it has a chain of two proper cyclic flats. Conversely, the V\'amos matroid is split (it is in fact sparse paving) but it is not a lattice path matroid as it is not representable over any field. By using the notion of stressed subset relaxations of \cite{ferroni-schroter} it should be possible to identify those matroids that are both lattice path and split matroids; this class strictly contains all \emph{cuspidal matroids} (see \cite[Proposition 3.24]{ferroni-schroter}).
\end{remark}

\section*{Acknowledgements}
The authors thank the reviewers of this article for their useful comments. LF is supported by the Swedish
Research Council grant 2018-03968. BS thanks the Knut and Alice Wallenberg Foundation for the support.

\appendix 
\section{Sage Code}\label{appendix}
The purpose of this appendix is to provide the code of two simple functions that allow to verify Lemma~\ref{lemma:rank2} for matroids with a small ground set using Sage \cite{sagemath}.  
\begin{lstlisting}
# The following routine produces the list of all loopless
# matroids of rank 2 with ground set [n] (up to isomorphism).
# Note that the set of hyperplanes of a loopless rank 2 matroid 
# forms a partition of the ground set
def all_rank_2_matroids(n):
    ans = []
    for P in Partitions(n):
        if(len(P) >= 2):
            circs = []
            i = 1
            for p in P:
                circ = list(range(1, n + 1))
                for j in range(i, i + p):
                    circ.remove(j)
                i = i + p
                circs.append(circ)
            ans.append(Matroid(circuits=circs).dual())
    return ans

# The following verifies that all loopless 
# and coloopless matroids of rank 2 with ground set [n] 
# (up to isomorphism) satisfy the multiplicative Merino-Welsh conjecture 
def check_mw_rank2(n):
    mat = all_rank_2_matroids(n)
    for M in mat:
        if(len(M.coloops()) > 0):
            continue
        else:
            T = M.tutte_polynomial()
            d = T(2,0) * T(0,2) - T(1,1)^2
            if(d < 0):
                return false
            if(d == 0):
                print("Matroids with equality: ", M)
    return true
    
# Testing the conjecture for rank 2 matroids of size at most 12
for n in range(2, 13):
    print("size ", n, " Merino-Welsh conjecture ", check_mw_rank2(n))

size  2   Merino-Welsh conjecture  True
size  3   Merino-Welsh conjecture  True
Matroids with equality:  Matroid of rank 2 on 4 elements with 4 bases
size  4   Merino-Welsh conjecture  True
size  5   Merino-Welsh conjecture  True
size  6   Merino-Welsh conjecture  True
size  7   Merino-Welsh conjecture  True
size  8   Merino-Welsh conjecture  True
size  9   Merino-Welsh conjecture  True
size  10   Merino-Welsh conjecture  True
size  11   Merino-Welsh conjecture  True
size  12   Merino-Welsh conjecture  True
\end{lstlisting}

\bibliographystyle{abbrv}
\bibliography{bibliography}

\end{document}